\documentclass[a4paper,12pt,reqno]{amsart}
\usepackage[centertags]{amsmath}
\usepackage{amsfonts,amssymb,amsthm} 
\usepackage[dvips]{graphicx} 
\usepackage{psfrag}
\usepackage[english]{babel}
\usepackage{newlfont}
\usepackage{color}
\usepackage[body={16cm,23cm},centering]{geometry} 
\usepackage{fancyhdr}
\usepackage[active]{srcltx}

\newtheorem{theorem}{Theorem}[section]
\newtheorem{lemma}[theorem]{Lemma}

\theoremstyle{definition}
\newtheorem{definition}[theorem]{Definition}

\newtheorem{rem}[theorem]{Remark}

\numberwithin{equation}{section}

\newcommand{\R}{\mathbb{R}}
\newcommand{\N}{\mathbb{N}}

\newcommand{\eps}{\varepsilon}

\newcommand{\diver}{\operatorname{div}}

\renewcommand{\MR}[1]{\null}

\newcommand{\Haus}{\mathcal{H}}

\begin{document}

\title[On representation of boundary integrals involving mean curvature]{On representation of boundary integrals involving the mean curvature for mean-convex domains}
\author{Yoshikazu Giga}
\address{Graduate School of Mathematical Sciences \\ University of Tokyo \\Komaba 3-8-1, \-Meguro-ku \\ Tokyo 153-8914, Japan}
\email{labgiga@ms.u-tokyo.ac.jp}

\author{Giovanni Pisante}
\address{Dipartimento di Matematica e Fisica \\ Seconda Universit\`a degli Studi diNapoli \\ Viale Lincoln, 5 \\ 81100 Caserta, Italy}
\email{giovanni.pisante@unina2.it}

\maketitle

\begin{abstract}
Given a mean-convex domain $\Omega\subset \R^n$ with boundary of class $C^{2,1}$, we provide a representation formula for a boundary integral of the type
\[
\int_{\partial \Omega} f(k(x)) \, d\mathcal{H}^{n-1}
\]
where $k\geq 0$ is the mean curvature of $\partial \Omega$ and $f$ is non-increasing and sufficiently regular, in terms of volume integrals and defect measure on the ridge set.
\end{abstract}

\section{Introduction}

In this note we are interested in giving an explicit representation for a particular class of curvature depending integral functionals defined on compact manifolds without boundary. We restrict our analysis to $C^{2,1}$ regular manifolds with non-negative mean curvature that can be identified as boundaries of mean-convex domains. More precisely, for a mean-convex domain $\Omega \subset \R^n$, denoted by $k(x)$ the mean curvature of $\partial \Omega$, we are interested in recovering the value of boundary integrals of the type
\begin{equation}\label{integral}
\int_{\partial \Omega} f(k(x)) \, d\mathcal{H}^{n-1}
\end{equation}
in terms of the behavior of $f$ inside $\Omega$ when $f$ satisfies suitable regularity assumptions. It turns out that, if $f$ is a differentiable non-increasing function (cf. Theorem \ref{mainthm} for the precise assumptions), \eqref{integral} can be expressed as the sum of volume integrals plus a defect measure $\delta^f$ concentrated on the ridge set of $\Omega$ as follows
\begin{equation}\label{formula}
\int_{\partial \Omega} f(k(x)) \,d \Haus^{n-1} =  \int_{\Omega} k(x) f(k(x))\,dx - \int_{\Omega} f'(k(x))|D^{2}d(x)|^{2} \, dx + \delta^{f}(\Sigma)
\end{equation}
where with $d$ and $\Sigma$ we have denoted the distance function from $\partial \Omega$ and its singular set respectively.

Our initial motivation for this study come from the understanding of a useful consequence of the area formula and Fubini's theorem, which is sometimes referred to as Heintze-Karcher's inequality (see \cite{HeiKar78-000} and  \cite[Theorem 6.16]{MonRos05-000}). In our setting it provides an upper bound of the measure of $\Omega$ in terms of the boundary integral of $1/k$. Indeed, it states that for any regular strictly mean-convex domain $\Omega\subset \R^{n}$ we have
\begin{equation}\label{HK-ineq}
\cal L^n (\Omega) \leq \frac{1}{n}\int_{\partial \Omega} \frac{1}{k(x)}\, d \cal{H}^{n-1},
\end{equation}
where $\cal L^n$ denotes the $n$-dimensional Lebesgue measure. The case of equality in \eqref{HK-ineq} is particularly interesting, as is known that equality occurs if and only if $\Omega$ is a ball and this property, combined with Minkowski integral formula, gives an elegant proof of the classical Alexandrov's theorem that identifies the sphere as the unique compact connected surface with constant mean curvature (see.  \cite[Theorem 6.17]{MonRos05-000}). This point of view has been recently used in \cite{He:2009p175} to prove an anisotropic version of Alexandrov's theorem for compact embedded hypersurfaces with constant anisotropic mean curvature.

Our goal was to obtain a sharp estimate for the error term in \eqref{HK-ineq}, 
and this led us to the study of the boundary integral 
\begin{equation*}\label{1suf}
\int_{\partial \Omega} f(k(x)) \, d \cal{H}^{n-1}.
\end{equation*}

An application of \eqref{formula} with $f(t)=1/t$ allows us to write
\[
\int_{\partial \Omega} \frac{1}{k(x)} \,d \Haus^{n-1} - n \cal{L} (\Omega)= \int_{\Omega} \frac{|D^{2}d(x)|^{2}-(n-1)k^2(x)}{k^2(x)} \, dx + \delta^{f}(\Sigma) \geq 0
\]
filling the gap in \eqref{HK-ineq} by explicitly expressing the error term as a sum of a volume integral and a defect measure in the spirit of  \cite{Aviles:1996p212}.

The idea behind the proof of the main result can be easily explained by a formal application of the divergence theorem. Indeed, recalling that, denoted by $\nu(x)$ the unit inner normal to $\partial \Omega$ at a point $x$, we have $\nu(x)=\nabla d(x)$, the divergence theorem should give us
\[
\int_{\partial \Omega} f(k) d \cal{H}^{n-1} = \int_{\partial \Omega}f(k)\nabla d \cdot \nu \,d \cal{H}^{n-1} \approx - \int_\Omega \diver \left( f(k)\nabla d \right) .
\]
The meaning of the last integral in the previous formula has to be clarified, since the integrand term a priori is just defined as a distribution. This is a key step of the proof of \eqref{formula} in Theorem \ref{mainthm}, whose main ingredient is indeed the identification of $-\diver\left( f(k)\nabla d \right)$ as a non-negative Radon measure which is absolutely continuous with respect to the $(n-1)$-dimensional Hausdorff measure. The formula then follows once we identify the densities of its Lebesgue decomposition. 

The paper is structured as follows. In the next section we recall some preliminary results on the regularity of the distance function and on functionals of measures. In Section \ref{sec-meas} we prove that $-\diver\left( f(k)\nabla d \right)$ is indeed a non-negative Radon measure in a neighborhood of $\Omega$. The proof of the main result is presented in Section \ref{sec-main}. In the last section the anisotropic version of the representation formula \eqref{formula} is briefly discussed. 
\section{Preliminaries}

We start recalling some classical regularity results concerning the distance function from the boundary of a domain that have been proved in  \cite{Li:2005p201,Li:2005p47}. Let $\Omega$ be a domain in $\R^n$. We define the singular set $\Sigma$ as the complement of the open set $G$ defined as the largest open subset of $\Omega$ such that for every $x\in G$ there is a unique closest point on $\partial \Omega$ from $x$. The following regularity result holds true.
\begin{theorem}[Li-Nirenberg] \label{lini-01}
Suppose that $\partial \Omega$ is of class $C^{\ell,\alpha}$, with $\ell\geq 2$ and $0<\alpha\leq 1$, then the distance function belongs to $C^{\ell,\alpha}(G\cup \partial \Omega)$. 
\end{theorem}
Let us comment on the properties of the singular set $\Sigma = \Omega\setminus G$ of the distance function to the boundary $\partial \Omega$ (sometimes called also {\it ridge} of $\Omega$ or {\it medial axes}). It is known that $\Sigma$ is always a connected set and has finite $(n-1)$-dimensional Hausdorff measure, in particular $\Sigma$ is a $(n-1)$-rectifiable set of finite perimeter. Another useful property is resumed here. Consider $y\in \partial \Omega$, the \emph{cut point} of $y$, denoted by $m(y)$ is defined as the point when, moving along the inner normal from $y$, the set $\Sigma$ is hit for the first time.
\begin{theorem}\label{lini-02}
Let $\Omega$ be of class $C^{2,1}$. Then from every point $y\in \partial \Omega$, the length $s(y)$ of the segment joining $y$ and its cut point $m(y)$ is Lipshitz continuous in $y$. 
\end{theorem}
As pointed out in \cite[Remark 1.2]{Li:2005p201} the regularty condition $C^{2,1}$ is sharp for the validity of the previous result.

\begin{definition}[mean-convex domain]\label{meanconvex} Let $\Omega$ be a bounded domain in $\mathbb{R}^n$ with boundary $\partial \Omega$ of class $C^{1,1}$ . We let $\nu$ be the interior unit normal to $\partial \Omega$ and $\mathbf{H}_{\partial \Omega}$ the mean curvature vector of $\partial \Omega$. We say that $\Omega$ is mean-convex if $\mathbf{H}_{\partial \Omega}$ is pointing inside $\Omega$ at every point, i.e. $\mathbf{H}_{\partial \Omega} \cdot \nu \geq 0$.
\end{definition}

In the sequel we will assume that $\Omega$ is a mean-convex domain. Since the boundary of $\Omega$ is supposed to be of class $C^2$ this is equivalent to require that the mean curvature $k(y)$ is non-negative for any $y\in\partial \Omega$.

Denote by $d(x)$ the distance function form $\partial \Omega$. We will write, with a slight abuse of notation, $k(x)=-\Delta d(x)$ and $\nu(x)=\nabla d(x)$. One of the steps needed for the proof of the identity \eqref{formula}  is the fact that the distribution $S=-\diver (f(k)\nabla d)$ is a non-negative Radon measure if $\Omega$ is mean-convex. Let $S$ be a real distribution, i.e. $S(\phi)$ is real for any real valued $\phi\in \mathcal{D}$, we say that $S$ is non-negative if $S(\phi)\geq 0$ for any $\phi \in \mathcal{D}$ with $\phi \geq 0$. We will make use of the following well known result (see \cite[Th\'eor\`eme V]{Schwartz:un}). 
\begin{theorem}\label{Measure}
A non-negative distribution $S$ can be identified as a continuous linear form on $\mathcal{D}^k$ equipped with the topology induced by $C^k$ ($k\geq 0$). Moreover $S$ is identified with a non-negative Radon measure $\mu_S$ through the equality
\[
S(\phi) = \int \phi\, d\mu_S.
\]
\end{theorem}

Here we recall some notations on functionals of measures, we refer to \cite{Demengel:1984cx} and \cite{Goffman:1964dj} for further details. In order to understand the behavior of continuous functions at infinity (needed to deal with functions of Radon measures) we will use the following notion. Let $f:\R^m\to \R\cup \{\infty\}$ with $f(0)< +\infty$, we define its \emph{recession function} as
\[
f^\infty(p):=\lim_{t\nearrow \infty} \frac{f(t p)}{t}. 
\]

The recession function can be used to give a meaning to functions depending on pairs of Radon measures. Let $ f:\R^m \to [0,\infty]$ be continuous and such that $f^\infty$ is well defined. The recession function is positively homogeneous of degree $1$, it is finite along the direction of (at most) linear growth of $f$ and is infinite along the direction of superlinear growth. Moreover it is clear that if $\|f\|_\infty < \infty$ then $f^\infty=0$ identically. Let $\mu$ and $\lambda$ be respectively an $\R^m$-valued and a positive measure in $\Omega \subset \R^n$, we can define the measure
\[
G(\mu,\lambda) := f\left(\frac{\mu}{\lambda}\right) \lambda + f^\infty\left(\frac{\mu^s}{|\mu^s|}\right) |\mu^s|\,,
\]
i.e. for a measurable set $E$
\[
G(\mu,\lambda)(E) := \int_{E }f\left(\frac{\mu}{\lambda}(x)\right) d \lambda(x) + f^\infty\left(\frac{\mu^s}{|\mu^s|}(x)\right) |\mu^s|(x),
\]
where $\mu^s$ is the singular part of $\mu$ with respect to $\lambda$ and $ \frac{\mu}{\lambda}$ denotes the Radon-Nikodym derivative of $ \mu $ with respect to $ \lambda $. 
When $\lambda$ is the Lebesgue measure we will indicate the measure $G(\mu,\lambda)$ with $f(\mu)$. It is worth to observe that if $f$ is a convex and lower semicontinuous function, then the recession function is always well defined and the functional $G$ turns out to be lower semicontinuous with respect to the weak-$*$ convergence of measures (cf. \cite[Theorem 2.34]{AmbFusPal00-000}). 

\begin{lemma}
Let $f$ be a real-valued bounded continuous function in $(0,+\infty)$ and let $\mu$ be a positive Radon measure. Then $f(\mu)$ is bounded, i.e. $f(\mu)$ is absolutely continuous with respect to $\mathcal{L}^n$ and its density is an $L^\infty$ function.
\end{lemma} 
\begin{proof}
Using the previous notation we have
\[
f(\mu) := f\left(\frac{\mu}{\mathcal{L}^n}\right) \mathcal{L}^n + f^\infty\left(\frac{\mu^s}{|\mu^s|}\right) |\mu^s|.
\]
Now we observe that $f^\infty$ is identically zero by the boundedness assumption on $f$. This imply the desired result since we can write
\[
f(\mu) := f\left(\frac{\mu}{\mathcal{L}^n}\right) \mathcal{L}^n
\] 
and $\| f\left(\frac{\mu}{\mathcal{L}^n}\right)\|_\infty \leq \| f\|_\infty $ .
\end{proof}

\section{The measure $-\diver (f(k)\nabla d)$}\label{sec-meas}

In this section we prove that $-\diver (f(k)\nabla d)$ is a non-negative Radon measure. The first step is to prove that $k:=-\Delta d$ is a positive Radon measure. We will use the \emph{normal coordinates} to represent points in $\Omega\setminus \Sigma$. For $z \in \Omega \setminus \Sigma$ we can write
$z = x + t \nu_x$ for a unique $x \in \partial \Omega$ and $0 < t < s(x)$, where $s(x) = \langle m(x) - x,
\nu_x\rangle$ and $\nu_x$ is the unit inner normal to $\partial \Omega$ at $x$. We recall that by Theorem \ref{lini-02} the function $s(x)$ is Lipschitz continuous. Moreover from Theorem \ref{lini-01} we know that $-\Delta d(z)$ is well-defined in $ G \cup \partial \Omega $. We can explicitly compute it in terms of the principal curvatures $k_r(x)$ of $\partial \Omega$ in the point $x$, we have indeed (see for instance \cite[Section 14.6, Lemma 14.17]{Gilbarg:2001p77})
\begin{equation}\label{dlaplacian}
-\Delta d(z) = \sum_{r=1}^{n-1} \frac{k_r(x)}{1- t k_r(x)}\, .
\end{equation}

\begin{lemma}Let $ \Omega $ be a mean-convex domain with boundary of class $ C^{2,1} $. Then $ -\Delta d $ is a non-negative Radon measure. If $ k(x) \geq c > 0 $ for any $x\in \partial \Omega$, then $-\Delta d$ is a positive Radon measure uniformly bounded from below, in particular $-\Delta d\geq c $ in distributional sense. 
\end{lemma}

\begin{proof}
Consider $S:=-\Delta d$ which is well-defined as a distribution on $\Omega$ since $\|\nabla d\|_\infty\leq 1$. Let $\varphi \geq 0$ be a test function. We use \eqref{dlaplacian} to infer, by explicit calculations, that if $\varphi$ is supported in $\Omega\setminus \Sigma$ we have $S(\varphi)\geq c$.  
Let $\theta$ be a cut-off function such that $\theta = 1$ on $[0, 1]$, $\theta = 0$ on $[2, \infty)$ and $\theta$ is non-increasing $C^1$ function. Define for $\varepsilon > 0$ 
\[
\psi_\varepsilon(x,t)=\theta \left(\frac{s(x) - t}{\varepsilon}\right)\;\; \textrm{ with }\;\;\varepsilon < \inf_{x\in \partial \Omega} \frac{s(x)}{2}.
\] 
We note that, by definition, 
\begin{equation}\label{positive}
\langle \nabla \psi_\varepsilon, \nabla d\rangle = \frac{\partial  \psi_\varepsilon}{\partial t}  \geq 0.
\end{equation} 
We use this cut-off to write $S(\varphi) = S(\psi_\varepsilon \varphi) + S((1 - \psi_\varepsilon) \varphi)$. As noted before, since $(1 - \psi_\varepsilon) \varphi$ is supported in $\Omega\setminus \Sigma$ we have $S((1 - \psi_\varepsilon) \varphi)\geq c$. For the other term we write, using \eqref{positive}
\[
S(\psi_\varepsilon \varphi)= \int_{\Omega} \varphi \, \langle \nabla \psi_\varepsilon , \nabla d \rangle + \int_{\Omega} \psi_\varepsilon \, \langle \nabla \varphi, \nabla d \rangle \geq  \int_{\Omega} \psi_\varepsilon \, \langle \nabla \varphi, \nabla d \rangle.
\]
Here and hereafter we often suppress $dx$ unless confusion occurs.
Since the last term tends to zero for $\varepsilon$ going to zero we get $S(\varphi)\geq 0$. The claim follows from the Theorem \ref{Measure}.

\end{proof}

\begin{rem}\label{bound}From the previous two lemmata we easily infer that if $f$ is bounded and if $\Omega$ is mean-convex, then the distribution $f(k)\nabla d$ can be identified with an essentially bounded map in $\Omega$. 
\end{rem}

Using the same idea of the previous lemma we are able now to prove the following.

\begin{theorem}\label{thm-measurediv} Let $\Omega$ be a mean-convex domain with boundary of class $C^{2,1}$ and let $f$ be a non-negative and non-increasing function of class $C^1$ in a neighborhood of the interval $[\min_{\partial\Omega}k(x), \infty)$, then the distribution \[S=-\diver [f(k)\nabla d]\] is a non-negative Radon measure in a sufficiently small neighborhood of $\overline{\Omega}$.
\end{theorem}
\begin{proof} We start by observing that the regularity of $\partial \Omega$ allows us to extend in a $C^{2,1}$ way the distance function in a $\delta$-neighborhood of $\overline{\Omega}$ defined by 
\[
\Omega_\delta:= \{x\in\R^n \;:\; d(x, \Omega) < \delta \}.
\] 
By the regularity assumptions on $f$, we can choose $\delta$ sufficiently small in order to have $f(k(y))\leq c < \infty$ for any $y\in \Omega_\delta\setminus \Omega$. We will continue to denote by $d(x)$ the extended distance function from $\partial \Omega$ defined in $\Omega_\delta$.  
Now we observe that if $\varphi$ is a test function supported in $\Omega_\delta \setminus \Sigma$ with $\varphi \geq 0$, then $S(\varphi) \geq 0$. This follows by direct calculations. Indeed, since in $\Omega_\delta \setminus \Sigma$ the distance function is $C^{2,1}$, we can differentiate two times the equality $|\nabla d |=1$ to get (using Einstein notation of summated indices) 
\[
\partial_{lji}d \,\partial_{i}d + \partial_{ji}\, \partial_{li}d = 0, 
\]
that implies, for $l=j$
\begin{equation}\label{secondform}
\langle\nabla d, \nabla k \rangle = - (\nabla d \cdot \nabla) \diver \nabla d =  |D^{2}d|^{2}
\end{equation}
and consequently
\[
\diver(f(k) \nabla d) = f'(k)\langle \nabla d , \nabla k \rangle - f(k) k \leq 0.
\]
Hence the last inequality follows as in the previous proof, observing that by \eqref{dlaplacian} the function $k$ is uniformly bounded by a positive constant in any compact set far from $\Sigma$.  

For a general $\varphi$ we can use the same cut-off procedure used in the proof of the previous lemma. We can therefore write as before
\[
S(\varphi) = S(\psi_\varepsilon \varphi) + S((1 - \psi_\varepsilon) \varphi).
\]
We observed already that the second term is positive due to the aforementioned explicit calculation. It remains to bound the quantity
\[
S(\psi_\varepsilon \varphi)= \int_{\Omega_\delta} \varphi \,f(k) \, \langle \nabla \psi_\varepsilon , \nabla d \rangle + \int_{\Omega_\delta}  \psi_\varepsilon \, \langle \nabla \varphi, f(k) \nabla d \rangle \geq  \int_{\Omega_\delta} \psi_\varepsilon \, \langle \nabla \varphi, f(k) \nabla d \rangle,
\] 
where the last inequality is a consequence of \eqref{positive} and of the non-negativity of the function $f$. Finally we observe that Remark \ref{bound}  and the definition of $\psi_\varepsilon$ allow us to infer that 
\[
\lim_{\varepsilon \to 0} \int_{\Omega_\delta} \psi_\varepsilon \, \langle \nabla \varphi, f(k) \nabla d \rangle = 0
\]
concluding the proof by Theorem \ref{Measure}.
\end{proof}

We observe now that the measure $\mu:=-\diver[f(k) \nabla d] $ is absolutely continuous with respect to $\Haus^{n-1}$. To this aim it is sufficient to bound the upper $(n-1)$-density of the measure $\mu$  (see \cite[Theorem 2.56]{AmbFusPal00-000})
\[
\Theta_{n-1}(\mu,x) = \limsup_{\rho\to 0} \frac{\mu(B_\rho(x))}{\rho^{n-1}} 
\]
for any $x\in \Omega$, where $B_\rho(x)$ denotes the closed ball of radius $\rho$ centered at $x$. This can be achieved using a smoothing argument and integration by parts. Indeed we recall that by Remark \ref{bound} $f(k)\nabla d$ is in $L^\infty(\Omega)$. Then if we consider a family of standard mollifiers $\rho_\varepsilon$ we have (up to subsequences) $|\rho_\varepsilon * f(k)\nabla d | \leq C $ for some $C< \infty$. Moreover, by standard properties of convolutions of measures (see \cite[Theorem 2.2]{AmbFusPal00-000}) we also have that $-\diver(f(k) \nabla d)  * \rho_\varepsilon$ is a $C^\infty$ function that locally weak-$*$ converges to $-\diver(f(k) \nabla d) $ in the sense of measures. Let $x \in \Omega$, we use the lower semicontinuity of the total variation with respect to weak-$*$ convergence of measures to ensure that (for sufficiently small $\rho$) we have
\[
\mu(B_\rho(x)) \leq \lim_{\varepsilon\to 0} \int_{B_\rho(x)}  -\diver(f(k) \nabla d)  * \rho_\varepsilon \, dx = -\int_{\partial B_\rho(x)} \langle f(k) \nabla d * \rho_\varepsilon , \nu \rangle \,d \Haus^{n-1} \leq C \omega_n \rho^{n-1}
\]
which proves the claim.

\section{Main result}\label{sec-main}
In this section we will precisely state and prove the main result of the paper. To this aim, we will need to know more precise properties of the ridge set. In particular, it will be useful to understand its {\it stratified} structure. The singular set of the distance function is characterized by the property that if $y\in \Sigma$ then there exist at least two points in $\partial \Omega$ in which the value $d(x):=d(x,\partial \Omega)$ of the distance function is attained. Let us consider the set $\Sigma_0\subset \Sigma$ of the points where the distance function form $\partial \Omega$ is attained exactly  in two points. Namely we set
\[
\Sigma_0:= \left\{  x\in \Sigma \;| \begin{array}{c} \text{there exists a unique pair} \, (y_1(x),y_2(x)) \in \partial \Omega \times \partial \Omega \\ \text{such that} \,y_1(x)\not=y_2(x)\,,\; d(x)=|x-y_1|=|x-y_2| \end{array} \right\}.
\]
For a fixed $x\in \Sigma_0$ let us denote by $\nu_1$ and $\nu_2$ the inner normal directions to $\partial \Omega$ at $y_1(x)$ and $y_2(x)$ respectively. We observe that, due to the expression \eqref{dlaplacian}, when moving along $\nu_i$ from $y_i$ toward $x\in \Sigma_0$, we can identify two limit values for the extended curvature function $k=-\Delta d$, namely for $j\in\{1,2\}$, 
\[
k^j(x)=\lim_{t\to s(y_j)} \sum_{r=1}^{n-1} \frac{k_r(y_j)}{1- t k_r(y_j)}.
\]

Our main result can be stated as follows.

\begin{theorem}\label{mainthm}
Let $\Omega$ be a mean-convex domain with boundary of class $C^{2,1}$ and let $f$ be a non-negative and non-increasing  function of class $C^1$ in a neighborhood of the interval $[\min_{\partial\Omega}k(x), \infty)$. Then the following formula holds
 \begin{equation}\label{general-iso}
\int_{\partial \Omega} f(k(x)) \,d \Haus^{n-1} =  \int_{\Omega} k(x) f(k(x))\,dx - \int_{\Omega} f'(k(x))|D^{2}d(x)|^{2} \, dx + \delta^{f}(\Sigma),
 \end{equation}
 where
  \[
 \delta^f(\Sigma)=\int_{\Sigma_{0}} \frac{|\nu_1(x)-\nu_2(x)|}{\sqrt{2}}\big[f(k^1(x))+f(k^2(x))\big] \,d \Haus^{n-1}(x).
 \]
\end{theorem}

Once we know that the distribution $S=-\diver(f(k)\nabla d)$ is a measure on $\Omega_\delta$ which is absolutely continuous with respect to the $(n-1)$-dimensional Hausdorff measure, a key step toward the proof of the main result is to identify the densities of its absolutely continuous part and of its singular part. The next lemma provides the desired identification.

\begin{lemma}\label{lemma-density} Let $f$ be as in Theorem \ref{mainthm}, then for the measure $-\diver(f(k)\nabla d)$ the following decomposition holds
\begin{equation}\label{represent}
\begin{split}
-\diver(f(k)\nabla d)  = &  \big(f(k) k -f'(k)|D^2 d|^2  \big) d\mathcal{L}^n \\
 & +  \left(\frac{|\nu_1(x)-\nu_2(x)|}{\sqrt{2}}\big[f(k^1(x))+f(k^2(x))\big] \right) \, \Haus^{n-1}|_{\Sigma_0}  \;.
\end{split}
\end{equation}

\end{lemma}

Before proving the Lemma \ref{lemma-density} we state two preliminary lemmata that will be useful throughout the proof. The first one is an orthogonality result of linear algebra.

\begin{lemma}\label{lemma-linear} 
Let $e_i$ for $i \in \{0, 1, 2\}$ be three different unit normals in $\R^n$. Then $e_i - e_0$ for  $i = 1, 2$ are linearly independent.
\end{lemma}

\begin{proof} 
Assume that $c_1$ and $c_2$ in $\R^n$ satisfy $c_1 (e_1 - e_0) + c_2 (e_2 - e_1) = 0$.
Thus $c_1 e_1 + c_2 e_2 = (c_1 + c_2) e_1$. If $c_1 + c_2$ is not zero, then $e_0$, $e_1$, $e_2$ lie on the same line which is
impossible since $e_i$ has length one. If $c_1 + c_2 = 0$, then $c_1 = 0$, $c_2
= 0$ since $e_1$ and $e_2$ are linearly independent. We thus conclude $c_1 =
c_2 =0$.
\end{proof}

The second one gives us some informations on the Jacobian of the Lipschitz map $m$ at regular points which are inverse images of conjugate points.
  
\begin{lemma}\label{jacob} Let $x \in \partial\Omega$ be such that $m$ and $s$ are differentiable at  $x$. Assume that 
\[
\lim_{t \to s(x)} \sum_{r=1}^{n-1} \frac{k_r(x)}{1- t k_r(x)}= \infty\;,
\]
then the Jacobian of $m$ at $x$ is zero.
\end{lemma}

\begin{proof} 

By assumption we know that there exist $r\in \{1,\dots n-1\}$ such that $s(x) = 1/k_r(x)$ and by definition we have
\[
m(x) = x+s(x) \nu(x),
\]
where $\nu(x)$ is the inner unit normal to $\partial \Omega$. Let $\tau_r(x)$ be the unit tangent vector at $x$ such that it is the eigenvector of the Weigarten map corresponding to the eigenvalue $k_r(x)$. Differentiating along the direction $\tau_r$ we get
\[
D_{\tau_r} m(x) = D_{\tau_r} s(x) \nu(x) + \big(1 - s(x)k_r(x)\big) \tau_r(x) =D_{\tau_r} s(x) \nu(s) .
\]

Assume that $D_{\tau_r} s$ is not zero at $x$. By the implicit function theorem we observe that $\Sigma$ is tangent to $\nu$ at $m(x)$. One often assumes $C^1$ regularity to get a $C^1$ implicit function. However, if one is interested in differentiability of the implicit function at $x$, the (Fr\'echet) differentiability of $m$ at $x$ is enough. Thus the set $\Sigma$ is tangent to the ray from $x$ to $m(x)$. This contradicts the fact that $s(x)$ is Lipschitz. Thus we conclude that $D_{\tau_r} m(x) = 0$. This implies that the Jacobian of $m$ at $x$ is zero.
\end{proof}

\begin{proof}[Proof of Lemma \ref{lemma-density}]
The proof will be divided in several steps.

\bigskip
{\sc{Step 1 :}}  First we justify the expression of $\delta^f(\Sigma)$ analyzing the local representation of $-\diver(f(k)\nabla d)$ near $\Sigma_0$. Let $U\subset \Omega$ be an open set such that $U\cap \Sigma_0 \not = \emptyset$ and there exist two connected, mutually disjoint and relatively open sets $\mathcal{N}_1$ and $\mathcal{N}_2$ on $\partial \Omega$ with the property that
\begin{equation}\label{surface}
U\cap \Sigma_0 = \{ x\in U \;| \; d(x,\mathcal{N}_1)-d(x,\mathcal{N}_2) = 0 \}.
\end{equation}
We use the notation $d_i(x)=d(x,\mathcal{N}_i)$ for $i=\{1,2\}$. Suppose moreover that $k\leq M <\infty$ in $U$. Under this condition, by the regularity of $\partial \Omega$ we can write $U=U_1\cup U_2 \cup (\Sigma_0\cap U)$, where $\Sigma_0\cap U$ is a $C^{2,1}$ hypersurface and 
\[
U_1:= \{ x\in U \;|\; d_1(x)< d_2(x) \} \;\;,\;\;\;\; U_2:=\{x\in U \;|\;  d_1(x)> d_2(x) \}. 
\]
Let $\phi$ be a test function supported in $U$, we have
\[
\begin{split}
\langle -\diver(f(k)\nabla d), \phi \rangle =  & \int_{U}  f(k)\nabla d \cdot \nabla \phi \, dx  =  \int_{U_1}  f(k)\nabla d \cdot \nabla \phi \, dx +  \int_{U_2}   f(k)\nabla d \cdot \nabla \phi \, dx \\
= & \int_{U_1}  f(-\Delta d_1)\nabla d_1 \cdot \nabla \phi \, dx +  \int_{U_2}   f(-\Delta d_2)\nabla d_2 \cdot \nabla \phi \, dx \\
= & - \int_{U_1}  f'(-\Delta d_1)|D^2 d_1|^2  \phi \, dx -  \int_{U_2}   f'(-\Delta d_2)|D^2 d_2|^2  \phi \, dx \\
& -  \int_{U_1}  f(-\Delta d_1) \Delta d_1  \phi \, dx -  \int_{U_2}   f(-\Delta d_2) \Delta d_2   \phi \, dx \\
& + \int _{\partial U_1} \phi f(-\Delta d_1)\nabla d_1 \cdot \nu_1 \, d\Haus^{n-1}  + \int _{\partial U_2} \phi f(-\Delta d_2)\nabla d_2 \cdot \nu_2 \, d\Haus^{n-1} \\
= & - \int_{U}  f'(k)|D^2 d|^2  \phi \, dx + \int_{U} f(k) k\, \phi \,dx  \\ 
 & + \int_{\Sigma_0}  \frac{|\nu_1(x)-\nu_2(x)|}{\sqrt{2}}[f(k^1(x))+f(k^2(x))] \,\phi \,d \Haus^{n-1}(x) . 
\end{split}
\] 
We note explicitly that, in order to justify the previous calculations, we can extend up to the boundary, in a canonical way, all the quantities we are interested in, when considered separately on $U_1$ and $U_2$. In the last equality of the previous formula we have used \eqref{dlaplacian}, \eqref{secondform} and the fact that we can explicitly calculate the expression of the normal to the boundary $\partial U_i$ in the portion that lives on $\Sigma_0$ (which is the only one that gives a contribution to the integral since $\phi$ is supported in $U$). More precisely in view of \eqref{surface}, we have
\[
\nu_1 = \frac{\nabla d_1-\nabla d_2}{|\nabla d_1-\nabla d_2|}=-\nu_2\,, 
\]
from which it follows that
\[
\nabla d_i \cdot \nu_i = \frac{1-\nabla d_1 \cdot \nabla d_2}{\sqrt{2} \sqrt{1- \nabla d_1 \cdot \nabla d_2}} = \frac{|\nabla d_1-\nabla d_2|}{\sqrt{2}}\,.
\]
Note that along a geodesic line the gradient of the distance function points always in the same direction.

\bigskip

\noindent {\sc{Step 2 :}}  We give now an estimate of the size of the subset of $\Sigma \setminus \Sigma_0$ where the curvature remains bounded. For a given $z\in\Sigma$ we define 
\[
\kappa(z) := \sup\left \{ \lim_{t\to s(y)} \sum_{r=1}^{n-1} \frac{k_r(y)}{1- t k_r(y)}  \; | \; y \in \partial \Omega \;\text{ with } \; d(z)=|z-y| \right\}.
\]
We are interested in the points of $\Sigma$ where $\kappa$ is bounded. For $M>0$ define
\[
G_M := \left\{ z\in \Sigma \; | \begin{array}{l} \kappa(z)\leq M \; \text{and there are  } y_1(z),y_2(z),y_3(z) \in \partial \Omega \\ \text{such that } d(z)=|z-y_i(z)| \; \text{for any } i\in\{1,2,3\} \end{array}\right  \}.
\]
We claim that $\cal H^{n-1}(G_M) =0 $. The proof of the claim is based on the observation that for a fixed $z\in G_M$ and for any collection of neighborhoods $\cal N_i$ of $y_i(z)$ in $\partial \Omega$ such that $y_i(z) \not \in \cal N_j$ if $j\not=i$, we have that $z$ is contained in the intersection of the two hypersurfaces given by
\[
\Gamma_1 := \big\{x \in \Omega \;:\; d(x,\cal N_1)=d(x,\cal N_3) \big\} \;\; \text{ and }\;\;\; \Gamma_2 := \big\{x \in \Omega \;:\; d(x,\cal N_2)=d(x,\cal N_3) \big\} .
\]
It is easy to check, as in {\sc{Step 1}} that, for $i\in \{1,2\}$, the normal vector of $\Gamma_i$ is parallel to $\nabla d(x,\cal N_i)-\nabla d(x,\cal N_3)$. Lemma \ref{lemma-linear} ensures us that $\Gamma_1$ and $\Gamma_2$ are transversal and thus $\cal H^{n-2}(\Gamma_1 \cap \Gamma_2) < \infty$.
The claim will be proved once we know that we can reduce to consider just countably many hypersurfaces of the type $\Gamma_i$ to cover the set $G_M$. 

We start by defining for any $h\in \N$ the set
\[
G^h_M := \left\{ z\in G_M \; | \begin{array}{l} \text{there exist } \; y_1,y_2,y_3 \in \partial \Omega  \text{ with } d(z)=|z-y_i|  \\  \text{and } |y_i-y_j|>\frac{1}{h} \text{ for any }\,i,j\in \{1,2,3\} \end{array}\right \}.
\]
Choose $\rho_h >0$ such that for any $y\in \cal N_h(y_i(z))=B_{\rho_h}(y_i)\cap \partial \Omega$ we have, independently of $z \in G^h_M$,
\[
\sum_{r=1}^{n-1} \frac{k_r(y)}{1- s(y) k_r(y)} \leq M+1.
\]
Since the set
\[
L := \left\{ y\in\partial \Omega \;:\; \sum_{r=1}^{n-1} \frac{k_r(y)}{1- s(y) k_r(y)} \leq M+1 \right \} 
\]
is compact, we may find a finite collection of points $\{y_i\}_{i\in I}$ with $y_i \not \in \cal N_h(y_j)$ for $i\not = j$ and 
\[
L\subset \bigcup_{i\in I}\cal N_h(y_i).
\] 
For any pair of points $y_i, y_j$ with $i,j \in I$ we can consider the hypersurfaces 
\[
\Gamma_{i,j}  := \big\{x \in \Omega \;:\; d(x,\cal N_h(y_i))=d(x,\cal N_h(y_j)) \big\} 
\]   
As before, we have that $\Gamma_{i,l}$ and $\Gamma_{j,l}$ are transversal. Then the set $G^h_M$ is contained in the finite union of the intersections of the type $\Gamma_{j,l}\cap \Gamma_{i,l}$ for $i,j,l \in I$. The claim follows by noticing that $G_M \subseteq \bigcup_{h\in \N} G_M^h$.

\bigskip 
\noindent {\sc{Step 3 :}} We analyze now the measure of the set
\[
K:=\{z \in \Sigma \;:\; \kappa(z) = +\infty \}.
\]
We claim that $\cal H^{n-1}(K)=0$. Since $K$ is a subset of the image through the Lipschitz map $m$ of the set
\[
A_\infty:=\left\{ x\in \partial \Omega \; : \; \lim_{t \to s(x)} \sum_{r=1}^{n-1} \frac{k_r(x)}{1- t k_r(x)}= \infty  \right\},
\]
it is sufficient to prove that 
\begin{equation}\label{zeromeas}
\cal H^{n-1}\big(m(A_\infty)\big)=0.
\end{equation}
We first observe that at $\cal H^{n-1}$-a.e. $x$ in $A_\infty$ the mappings $m$ and $s$ are differentiable by the Rademacher's theorem. Then the 
equality \eqref{zeromeas} follows from Lemma \ref{jacob} and from the area formula for Lipschitz maps (cf. \cite[Theorem 2.71]{AmbFusPal00-000} or \cite[Theorem 3.2.3]{Fed69-000}). 

\bigskip 
\noindent {\sc{Step 4 :}} From {\sc{Step 3}}, since the measure $\mu:=-\diver(f(k)\nabla d)$ is absolutely continuous with respect to $\cal H^{n-1}$, we have that the set $K$ is negligible with respect to $\mu$. It is then enough to prove \eqref{represent} on the set $\Omega_\delta \setminus K$. The claim will follow by continuity of measures along increasing sequences of sets once we have proved that \eqref{represent} holds on the set
\[
W_M:= \{ x\in  \Omega_\delta \,:\, \kappa(x) < M\},
\]
for any $M>0$. To this aim we observe that, since $z\in \Sigma_0$ is a conjugate point only if $\kappa(z)=\infty$,  by \cite[Theorem 1] {Li:2005p47}, we can find an open covering of $\Sigma_0 \cap  W_M$ given by $\{U_i\}_{i\in I}$, with $U_i$ satisfying \eqref{surface} as in {\sc Step 1}. We can then consider  the open covering of $W_M\setminus G_M$ given by $U_0 \cup \bigcup_{i\in I}U_i$ where $U_0$ is such that
 $\partial \Omega \subset U_0$ and $U_0 \cap \Sigma = \emptyset$. 

If $\phi$ is a test function supported on $U_0$, we can apply the integration by parts to obtain
\begin{equation}\label{U1}
\langle -\diver(f(k)\nabla d), \phi \rangle = - \int_{U_0}  f'(k(x))|D^2 d(x)|^2  \phi(x) \, dx + \int_{U_1} f(k(x)) k(x)\, \phi(x) \,dx. 
\end{equation}
Finally, using a partition of unity subordinate to the cover $U_0 \cup \bigcup_{i\in I}U_i$, from \eqref{U1} and the {\sc Step 1}, we have the claim.  
\end{proof}

We are finally in position to prove the main theorem.

\begin{proof}[Proof of Theorem \ref{mainthm}] 
Let us consider the characteristic function of $\Omega$, $\chi_\Omega$ and approximate it using a standard mollifier with $\chi^\eps_\Omega =\rho_\eps * \chi_\Omega \in C^\infty_{0}(\Omega_{2\eps})$. We explicitly observe that $\nabla \chi^\eps_\Omega=0$ in $\Omega_{-2\eps}$ since $\chi^\eps_\Omega(x) = 1$ for any $x\in \Omega_{-2\eps}$. We can then compute
\begin{equation}\label{limit-01}
\begin{split}
\lim_{\eps\to 0}  \langle -\diver(f(k)\nabla d),  \chi^\eps_\Omega \rangle  = & \lim_{\eps\to 0}  \int_{(\partial \Omega)_\eps} f(k)\nabla d \cdot \nabla \chi^\eps_\Omega \, dx \\
 = & \lim_{\eps \to 0} \int_{(\partial \Omega)_\delta} f(k) \nabla d \cdot d D(\chi^\eps_{\Omega}) \\
= & \int_{(\partial \Omega)_\delta} f(k) \nabla d \cdot d D(\chi_{\Omega}) 	\\
= & \int_{\partial \Omega} f(k) \, d \Haus^{n-1},
\end{split} 
\end{equation}
where $(A)_\delta$ denotes the $\delta$-neighborhood of $A$. To conclude the proof it is sufficient to observe that from \eqref{represent} we also have
\[
\begin{split}
\lim_{\eps\to 0}  \langle -\diver(f(k)\nabla d),  \chi^\eps_\Omega \rangle  = & \lim_{\eps\to 0}  \int_{\Omega_\eps}  \big(f(k) k -f'(k)|D^2 d|^2 \big)  \, dx \\
  & + \int_{\Sigma_{0}} \frac{|\nu_1(x)-\nu_2(x)|}{\sqrt{2}}[f(k^1(x))+f(k^2(x))] \,d \Haus^{n-1}(x) \\
= &  \int_{\Omega}  \big( f(k(x)) k(x)-f'(k(x))|D^2 d(x)|^2 \big)  \, dx \\
  & + \int_{\Sigma_{0}} \frac{|\nu_1(x)-\nu_2(x)|}{\sqrt{2}}[f(k^1(x))+f(k^2(x))] \,d \Haus^{n-1}(x).
\end{split} 
\]
\end{proof}

\section{The anisotropic case}

In this section we briefly describe how our approach can be also used in the anisotropic setting to provide a representation for boundary integrals involving anisotropic curvatures. We restrict ourself here to a formal discussion implicitly assuming all the regularity properties and the necessary assumptions needed to perform the calculations, without mentioning them.

We start recalling some useful notations (cf. for instance \cite{Bellettini:2009p223}). Let $\varphi^0$ be a one-homogeneous convex function that will be our anisotropy and let $\varphi$ be its dual function. The anisotropic distance function of a
point $x$ from a set $C$ is
defined by
\[
d_\varphi(x,C)=\inf_{y\in C}\varphi(y-x).
\]
For a given domain $\Omega$ we consider the signed anisotropic distance function from
$\partial \Omega$, defined as
\[
d_\varphi(x)=d_\varphi(x,\Omega)-d_\varphi(x,\mathbb{R}^n\setminus \Omega).
\]
We denote by $n_\varphi$ the Cahn-Hoffman vector and $k_\varphi$ the
anisotropic mean curvature. More precisely, defined 
\[
T_{\varphi^0} := \frac{1}{2} \nabla \big( ( \varphi^0)^2 \big),
\]
we set 
\[
n^*_{\varphi}=\frac{n(x)}{\varphi^0(n(x))} \;\; ,\;\;\; n_\varphi = T_{\varphi^0} (n^*_{\varphi}).
\]
We have, near $\partial \Omega$, 
\[
n^*_{\varphi} = \nabla d_\varphi \;,\;\;
\langle n^*_{\varphi} , n_\varphi \rangle = 1 \;, \;\;
k_\varphi=-\operatorname{div} n_\varphi. 
\]

As in the isotropic case, a formal application of the Gauss theorem leads to
\[
\begin{split}
\int_{\partial \Omega} f(k_{\varphi})\varphi^{0}(n) \, d\cal H^{n-1} & = \int_{\partial \Omega}
f(k_{\varphi})\varphi^{0}(n)\langle n_\varphi^*,n_\varphi
\rangle \, d\cal H^{n-1}  \\ 
 & =\int_{\partial \Omega}  f(k_{\varphi})\,n_\varphi \cdot n  \, d\cal H^{n-1}  \approx 
-\int_\Omega \diver 
\left(
f(k_{\varphi})\,n_\varphi
\right).
\end{split}
\]
The main issue, as in the previous case, is to give a representation of the measure $- \diver 
\left(
f(k_{\varphi})\,n_\varphi
\right)$. Let us observe that, outside of the singular set of $d_\varphi$ we have
\[
\diver\left(f(k_{\varphi})\,n_\varphi \right) = f'(k_{\varphi}) \nabla k_{\varphi} \cdot n_\varphi -  f(k_{\varphi})k_{\varphi} \,;
\]
moreover, the term $ \nabla k_{\varphi} \cdot n_\varphi$ can be expressed in terms of the squared
norm of the $\varphi$-Weingarten operator. To this aim for instance we can consider a proper test function $\phi$ and compute
\begin{equation*}
  \begin{split}
  \int_\Omega (n_\varphi \cdot \nabla) \operatorname{div}n_\varphi \phi  &
  = - \int_\Omega \partial_j  (n_\varphi^i \phi)  \partial_i (n^j_\varphi) \\
  & = -
  \int_\Omega  \partial_j  (n_\varphi^i)  \partial_i (n^j_\varphi) \phi -
  \int_\Omega n_\varphi^i \partial_i (n^j_\varphi)  \partial_j \phi \\
 & = \int_\Omega \| \nabla n_\varphi\|^2 \phi, 
\end{split}
\end{equation*}
where $\nabla n_\varphi$ denotes the Jacobian matrix of $n_\varphi$ and $\|\nabla n_\varphi\|^2$ denotes the square of its Euclidean norm, i.e. the sum of the squares of the components of the matrix $\nabla n_\varphi$. In the last equality we have used the symmetry of the Jacobian of the Cahn-Hoffman vector and the Euler formula for homogeneous
functions. We explicitly note that $\| \nabla n_\varphi\|^2$ is the squared norm of the anisotropic Weingarten operator that correspond to $\| D^2 d\|^2$ in the isotropic setting (cf. Lemma 3.2 and Remark 3.3 in \cite{BelFra00-000}). 

Proceeding similarly as in the previous section, using $d_\varphi$ instead of $d$, we may derive the following representation formula
\[
\int_{\partial \Omega} f(k_{\phi})\varphi^{0}(n) d\cal H^{n-1}= \int_\Omega f'(k_{\varphi}) \| \nabla n_\varphi\|^2dx - \int_\Omega f(k_{\varphi})k_{\varphi}\,dx + \delta_\varphi^f(\Sigma_\varphi),
\]
where 
\[
\delta_\varphi^f(\Sigma_\varphi) = \int_{\Sigma_0^\varphi}\big( f(k_{\varphi}^1(x))n_\varphi(y_1)- f(k_{\varphi}^2(x))n_\varphi(y_2) \big) \frac{n^*_\varphi(y_1)-n^*_\varphi(y_2)}{|n^*_\varphi(y_1)-n^*_\varphi(y_2)|} d\cal H^{n-1}.
\]

\section{Acknowledgments} 
This work was initiated while the second author visited
the Department of Mathematics of Hokkaido University during 2010. Its
hospitality is gratefully acknowledged. This work was completed when the
second author visited the School of Engineering Science of Osaka
University during 2012/13 under Marie Curie project IRSES-2009-247486. The work of the
first author has been partly supported by the Japan Society for the
Promotion of Science (JSPS) through grants for scientific research Kiban
(S) 21224001 and Kiban (A) 23244015.

\bibliographystyle{plain}

\end{document}